\documentclass[a4paper,reqno,11pt]{amsart}

\usepackage[margin=1in]{geometry}
\usepackage{amsmath, amssymb,amsthm,url,mathrsfs,graphicx,amscd,amsfonts}
\usepackage[mathscr]{eucal}
\usepackage{dsfont}
\usepackage{mathtools}
\usepackage{hyperref}
\usepackage{bbm}
\usepackage{soul} 

\usepackage[utf8]{inputenc}
\usepackage{graphicx,todonotes,hyperref}
\usepackage{epstopdf}
\usepackage{inputenc}
\usepackage{amsmath}
\usepackage{xcolor}
\usepackage{xspace}
\usepackage{amssymb}
\usepackage{amsfonts}
\usepackage{graphicx}
\usepackage{mathtools}
\usepackage{amsthm}
\usepackage{tikz-cd}
\usepackage{hyperref}
\usepackage{multirow}
\usepackage{enumitem}

\input xy
\xyoption{all}

\newtheorem*{acknowledgements*}{Acknowledgements}

\newcommand{\h}{\mathbb{H}}
\newcommand{\s}{\mathbb{S}}
\newcommand{\z}{\mathbb{Z}}
\newcommand{\R}{\mathbb{R}}

\allowdisplaybreaks
\theoremstyle{definition}
\newtheorem{theorem}{Theorem}[section]
\newtheorem{prop}[theorem]{Proposition}
\newtheorem{lemma}[theorem]{Lemma}
\newtheorem{cor}[theorem]{Corollary}
\newtheorem{note}[theorem]{Note}
\newtheorem{defn}[theorem]{Definition}

\newtheorem{rmk}[theorem]{Remark}
\newtheorem{fact}[theorem]{Fact}

\newtheorem{thm}[theorem]{Theorem}
\newtheorem*{theorem*}{Theorem}
\newtheorem*{corollary*}{Corollary}
\newtheorem*{prop*}{Proposition}
\newtheorem*{rmk*}{Remark}
\newtheorem{manualtheoreminner}{Theorem}

\newenvironment{manualtheorem}[1]{%
	\renewcommand\themanualtheoreminner{#1}%
	\manualtheoreminner
}{\endmanualtheoreminner}

\begin{document}
	\title[]{Smooth Structures on the product of $3$-connected $8$-manifolds with spheres}
	\author{Ankur Sarkar}
	\address{The Institute of Mathematical Sciences, A CI of Homi Bhabha National Institute, CIT Campus, Taramani, Chennai 600113, India.}
	\email{ankurs@imsc.res.in; ankurimsc@gmail.com}	
	
	\begin{abstract}
 Let $M$ be a closed, $3$-connected, $8$-dimensional smooth manifold. In this paper, we compute the concordance inertia group of the product manifold $M\times\mathbb{S}^k$ for $1\leq k\leq 14$ and classify all smooth manifolds homeomorphic to $M\times\mathbb{S}^k,$ up to concordance for $1\leq k\leq 10.$ Moreover, we provide a diffeomorphism classification of smooth manifolds homeomorphic to $M\times\mathbb{S}^1,$ where $H^4(M;\mathbb{Z})=\mathbb{Z}.$ 
	\end{abstract}
        \maketitle

	\section{Introduction} 
Let $\Theta_n$ denote the set of oriented diffeomorphism classes of homotopy $n$ -spheres, i.e., manifolds homeomorphic to the standard $n$-sphere $\mathbb{S}^n$ for $n > 4$. In \cite{milnor}, Kervaire and Milnor proved that $\Theta_n$ forms a finite abelian group under the connected sum operation, with the equivalence class of the standard $n$-sphere $\mathbb{S}^n$ as the zero element. Any homotopy $n$-sphere representing a nonzero element in $\Theta_n$ is called an exotic $n$-sphere.

Let $M$ be a closed, oriented, smooth $n$-dimensional manifold. One way to modify its smooth structure without changing its homeomorphism type is to take its connected sum with an exotic sphere $\Sigma^n$. This construction often changes the diffeomorphism type of $M$, but this is not always the case.

The \textit{inertia group} of $M$, denoted $I(M)$, is defined as the subgroup of $\Theta_n$ consisting of all homotopy spheres $\Sigma^n$ such that $M \# \Sigma^n$ is oriented diffeomorphic to $M$. Computing $I(M)$ explicitly and establishing bounds for it is generally challenging and has been achieved only in special cases (see \cite{kubo, schultz, BK18, BK22} and \cite{Kosinski, DCCN20, SZ22}).

However, within the inertia group, certain subgroups are of particular interest in the classification of smooth structures. One such subgroup is the \textit{homotopy inertia group}, denoted $I_h(M)$, which consists of those homotopy $n$-spheres $\Sigma^n$ for which there exists a diffeomorphism $M \# \Sigma^n \to M$ that is homotopic to the canonical homeomorphism. Another significant subgroup is the \textit{concordance inertia group}, denoted $I_c(M)$, containing all homotopy $n$-spheres $\Sigma^n$ such that $M \# \Sigma^n$ is concordant to $M$ (see \cite{JM70, JL70}). Clearly, $I_c(M)\subseteq I_h(M)\subseteq I(M) \subseteq\Theta_n.$

For certain product manifolds, these groups have been explicitly computed, such as $\mathbb{S}^p \times \mathbb{S}^q$ \cite{schultz}, $\mathbb{C}P^2 \times \mathbb{S}^7$ \cite{bel}, $\mathbb{C}P^3 \times \mathbb{S}^3$ \cite{XW21}, and $M \times \mathbb{S}^k$ (for $1 \leq k \leq 10$) where $M$ is a closed, oriented 4-manifold or a simply connected 5-manifold \cite{SBRKAS}. In this paper, we focus on the classification of smooth structures on $M \times \mathbb{S}^k$, where $M$ is a closed, $3$-connected, $8$-dimensional smooth manifold, and $k$ takes specific values. Manifolds of this type, $M$ (and more generally, $(n-1)$-connected $2n$-manifolds with $n \geq 3$), were classified by Wall in \cite{wall_highly_connected_manifold}, up to connected sum with a homotopy sphere, using their intersection form
\[ \lambda_M: H^4(M;\mathbb{Z}) \times H^4(M;\mathbb{Z}) \to \mathbb{Z} \]
and their stable tangential invariant
\[ \beta_M: H^4(M;\mathbb{Z}) \to \pi_3(\mathrm{SO}) \cong \mathbb{Z}, \]
which is defined by first realizing the cohomology classes of $H^4(M; \mathbb{Z})$ by embeddings of 4-spheres in $M$ and then taking the clutching functions of the normal bundles of the 4-spheres in $M$. The above result reduced the classification of 3-connected 8-manifolds $M$ to determining the inertia group $I(M)$. Recently, Crowley and Nagy \cite{DCCN20} announced a complete determination of the inertia groups $I(M)$.

Our first result focuses on using the stable homotopy type of $M$ to compute the concordance inertia group $I_c(M \times \mathbb{S}^k)$. Recall from \cite[Theorem 1.1]{rhuang} that if the fourth Betti number of \( M \) (referred to as the rank of \( M \)) is at least 2, then the stable homotopy type of \( M \) is determined by the number \( \mathrm{ind}(M) \). This number is defined as the index of the subgroup \( \mathrm{Im}(\chi) \) in \( \pi_3^s=\mathbb{Z}/{24}\{\nu\} \) when \( \operatorname{Im}(\chi) \neq 0 \), and is zero otherwise. The homomorphism \( \chi \) is given by the composition  
\[
\chi: H_4(M; \mathbb{Z}) \xrightarrow{\beta_M} \pi_3(SO(4)) \xrightarrow{J} \pi_7(S^4) \xrightarrow{E} \pi^s_3,
\]
where \( J \) is the classical \( J \)-homomorphism and \( E \) is the suspension map. The following theorem gives the concordance inertia groups of $M \times S^k$, in terms of $\mathrm{ind}(M)$ and the first Pontryagin class of $M$.

\begin{manualtheorem}{A}
\textit{Let $M$ be a closed, 3-connected, smooth 8-dimensional manifold of rank $m \geq 2$. Then:}
\begin{itemize}
    \item[(i)] $I_c(M\times\s^k)=0$ \textit{for} $1\leq k\leq 14,k \neq 5,9,13$ \textit{and} $14.$ 
    \item[(ii)] \textit{For} $k=5:$
    \begin{itemize}
       \item[(A)] $I_c(M\times\s^5)=0,$ \textit{if} $3\mid p_1(M).$ 
        \item[(B)] $I_c(M\times\s^5)=\mathbb{Z}/3,$ \textit{if} $3 \nmid p_1(M).$
    \end{itemize}
    \item[(iii)] \textit{For} $k=9,13,14$:
    \begin{itemize}
           \item[(A)] $I_c(M\times\mathbb{S}^k)=0,$ \textit{if} $\mathrm{ind}(M)$ \textit{is even}.
           \item[(B)] $I_c(M\times\mathbb{S}^k)=\mathbb{Z}/2,$ \textit{if} $\mathrm{ind}(M)$ \textit{is odd}.
    \end{itemize}
\end{itemize}
(Theorem A is a part of Theorem \ref{ConcorMtimesSk}.)
\end{manualtheorem} 
We note that every closed $(n-1)$-connected manifold of dimension \(2n\geq 4\) with rank 1 is a projective plane-like manifold. Recall that a \textit{projective plane-like manifold}  is a simply connected, closed topological manifold $M$ such that $H_*(M;\mathbb{Z})\cong \mathbb{Z}\oplus\mathbb{Z}\oplus \mathbb{Z}.$ Examples of such manifolds include the projective planes $\mathbb{C}P^2$, $\mathbb{H}P^2$, and $\mathbb{O}P^2$, which are defined over the complex numbers, quaternions, and octonions, respectively. Eells and Kuiper \cite{EK62} established several fundamental results in the study of projective plane-like manifolds. For instance, they showed that the integral cohomology ring of such a projective plane-like manifold $M$ is isomorphic to that of a projective plane, i.e., \[ H^*(M; \mathbb{Z}) \cong \mathbb{Z}[x]/(x^3).\] This result, combined with the solution of the Hopf invariant one problem, implies that the dimension of $M$ must be $4,8,$ or $16$ (cf. \cite[\S 5]{EK62}). Moreover, they showed that there are six homotopy types of projective plane-like manifolds of dimension $8$ and sixty of dimension $16$ \cite[\S 6]{EK62}. 

Furthermore, Eells and Kuiper classified smooth projective plane-like manifolds of dimensions $8$ and $16$ up to connected sum with homotopy spheres. This classification was later completed by Kramer and Stolz \cite{LKSS07}, who showed that the diffeomorphism class of a smooth projective plane-like manifold $M$ of dimension $2n \geq 8$ is determined by its Pontryagin number \[ p_{\frac{n}{4}}^2(M)[M] \in \mathbb{Z}.\] The results of Eells and Kuiper, combined with those of Wall \cite{wall_highly_connected_manifold}, established that an integer $k$ equals the Pontryagin number $p_{\frac{n}{4}}^2(M)[M]$ of such a manifold if and only if $k$ takes the form: \[ k = 2^2(1 + 2t)^2 \quad \text{with } t \equiv 0, 7, 48, 55 \pmod{56} \quad \text{(for } n = 4 \text{)} \] \[ k = 6^2(1 + 2t)^2 \quad \text{with } t \equiv 0, 127, 16128, 16255 \pmod{16256} \quad \text{(for } n = 8 \text{)} \] (see, for example, \cite[Theorem 1.3]{LKSS07}). These results provide an infinite family of smooth projective plane-like manifolds in dimensions $4$ and $8$  each of which possesses a unique differentiable structure, in the sense that any manifold homeomorphic to $M$ is, in fact, diffeomorphic to $M$. 
 
In the following, we describe the concordance inertia group for the product of a smooth projective plane-like $8$-dimensional manifold with the sphere $\mathbb{S}^k.$
\begin{manualtheorem}{B}
    Let $M$ be a smooth projective plane-like of dimension $8$. Then:
 \begin{enumerate}
    \item[(a)] $I_c(M\times\s^k)$ is zero for $1\leq k \leq 14$ and $k\neq 5,9, 13,14.$ 
     \item[(b)] for $k=5$
    \begin{itemize}
       \item[(i)] $I_c(M\times\s^5)=0,$ if $3\mid p_1(M).$ 
        \item[(ii)] $I_c(M\times\s^5)=\mathbb{Z}/3,$ if $3 \nmid p_1(M).$
    \end{itemize}
    \item[(c)] $I_c(M\times\mathbb{S}^k)=\mathbb{Z}/2$ for $k=9, 13,$ and $14.$
 \end{enumerate}
 (Theorem B is a part of Theorem \ref{ConcorMtimesSk}.)
\end{manualtheorem}
It follows from \cite{gw2, GB70} that there is a split short exact sequence
\[ 0 \to bP_{8k+2} = \mathbb{Z}_2 \to \Theta_{8k+1} \to \text{Coker}(J_{8k+1}) \to 0, \quad \text{for all } k \geq 1, \]
where \(bP_{8k+2}\) is the group of homotopy \((8k+1)\)-spheres that bound parallelizable manifolds, and \( \mathrm{Coker}(J_{8k+1}) = \pi^s_{8k+1}/{\mathrm{Im}(J)}.\) In particular, $\mathrm{Coker}(J_{9}) = \mathbb{Z}/2\{\nu^3\} \oplus \mathbb{Z}/2\{\mu\}$ and \(\mathrm{Coker}(J_{17}) = \mathbb{Z}/2\{\eta \eta^*\} \oplus \mathbb{Z}/2\{\nu \kappa\} + \mathbb{Z}/2\{\bar{\mu}\}\) \cite{ravenel, toda}.

As an initial step toward understanding the diffeomorphism types of $M\times\mathbb{S}^1$, we examine its inertia group, leading to the following result.
\begin{manualtheorem}{C}
    Let $M$ be a smooth projective plane-like manifold of dimension $2n\geq 8.$ 
    \begin{itemize}
        \item[(i)] If $n=4$, then $I(M\times\mathbb{S}^1)=bP_{10}\oplus\mathbb{Z}/2\{\Sigma_{\nu^3}\},$ where $\Sigma_{\nu^3}$ is the exotic $9$-sphere associated with the generator $\nu^3 \in \text{Coker}(J_{9}).$ 
        \item[(ii)] If $n=8,$ then $I(M\times\mathbb{S}^1)=bP_{18}\oplus\mathbb{Z}/2\{\Sigma_{\eta\eta^*}\},$ where $\Sigma_{\eta\eta^*}$ is the exotic $17$-sphere corresponding to the generator $\eta\eta^*\in \text{Coker}(J_{17}).$
    \end{itemize}
    (Theorem C is a combination of Lemma \ref{Inertiagrouphp2s1} and Lemma \ref{Inertiagroupop2s1}.)
\end{manualtheorem}
Let \(\mathcal{C}(M)\) denote the group of concordance classes of smooth structures on \(M\) (see Definition \ref{definition_CM}). Using Theorem A, Theorem B, and the stable decomposition of the closed, $3$-connected, $8$-dimensional smooth manifold $M$, we have the following explicit computations of the group $\mathcal{C}(M\times\mathbb{S}^k)$. 

 \begin{manualtheorem}{D}
      Let $M$ be a closed, $3$-connected, $8$-dimensional smooth manifold of rank $m\geq 1.$ Then:
      \begin{itemize}
          \item[(i)] $\mathcal{C}(M\times\mathbb{S}^k)= \Theta_{8+k}\oplus\Theta_8$ for $k=1$ and $2.$ 
          \item[(ii)] $\mathcal{C}(M\times\mathbb{S}^3)= \Theta_{11}\oplus\bigoplus\limits_{i=1}^m\Theta_{7}\oplus\Theta_8\oplus\mathbb{Z}/2.$
          \item[(iii)] $\mathcal{C}(M\times\mathbb{S}^4)=\bigoplus\limits_{i=1}^{m+1}\Theta_{8}.$
          \item[(iv)] $\mathcal{C}(M\times\mathbb{S}^k)=\Theta_{8+k}\oplus\bigoplus\limits_{i=1}^m\Theta_{4+k}\oplus\Theta_{8}\oplus\Theta_k$ for $k=7$ and $8.$ 
          \item[(v)] For $k=5:$
              \begin{itemize}
                  \item[(a)] $\mathcal{C}(M\times \mathbb{S}^5)=\Theta_{13}\oplus\bigoplus\limits_{i=1}^m \Theta_9\oplus\Theta_8,$ if $3\mid p_1(M).$
                  \item[(b)] $\mathcal{C}(M\times \mathbb{S}^5)=\bigoplus\limits_{i=1}^m \Theta_9\oplus\Theta_8,$ if $3\nmid p_1(M).$
              \end{itemize}
          \item[(vi)] For $k=6:$
              \begin{itemize}
                  \item[(a)] $\mathcal{C}(M\times\mathbb{S}^6)=G_1\oplus\mathbb{Z}/3\oplus\bigoplus\limits_{i=1}^{m-1}\Theta_{10}\oplus\Theta_8,$ when $3\mid p_1(M).$ 
                  \item[(b)] $\mathcal{C}(M\times\mathbb{S}^6)=G_2\oplus\bigoplus\limits_{i=1}^{m-1}\Theta_{10}\oplus\Theta_8,$ when $3\nmid p_1(M).$ 
              \end{itemize}
         where the groups $G_1$ and $G_2$ are extensions of the group $\mathbb{Z}/2\subset \Theta_{10}$ by the group $\Theta_{14}.$ 
          \item[(vii)] For $k=9:$
          \begin{enumerate}
              \item If $m=1,$ then $\mathcal{C}(M\times\mathbb{S}^9)=\bigoplus\limits_{i=1}^3\mathbb{Z}/2\oplus\Theta_{13}\oplus\Theta_8\oplus\Theta_9.$
              \item If $m\geq 2:$ 
              \begin{itemize}
                  \item[(a)] $\mathcal{C}(M\times\mathbb{S}^9)= \Theta_{17}\oplus\bigoplus\limits_{i=1}^m\Theta_{13}\oplus\Theta_8\oplus\Theta_9,$  for even $\textrm{ind}(M).$
                  \item[(b)]  $\mathcal{C}(M\times\mathbb{S}^9)=\bigoplus\limits_{i=1}^3\mathbb{Z}/2\oplus\bigoplus\limits_{i=1}^m\Theta_{13}\oplus\Theta_8\oplus\Theta_9,$ for odd $\textrm{ind}(M).$
              \end{itemize}
          \end{enumerate}
         \item[(viii)] For $k=10:$
          \begin{enumerate}
              \item If $m=1,$ then $\mathcal{C}(M\times\mathbb{S}^{10})=\Theta_{18}\oplus\Theta_8\oplus\Theta_{10}.$
              \item If $m\geq 2,$ then: 
              \begin{itemize}
                  \item[(a)]  when $\textrm{ind}(M)$ is even, then $\mathcal{C}(M\times\mathbb{S}^{10})= H\oplus\bigoplus\limits_{i=1}^{m-1} \Theta_{14}\oplus\Theta_{8}\oplus\Theta_{10},$ where the group $H$ is an extension of the group $\Theta_{14}$ by the group $\Theta_{18}.$
                  \item[(b)] When $\textrm{ind}(M)$ is odd, then $\mathcal{C}(M\times\mathbb{S}^{10})=\Theta_{18}\oplus\bigoplus\limits_{i=1}^{m-1}\Theta_{14}\oplus\Theta_8\oplus\Theta_{10}.$
              \end{itemize}
          \end{enumerate}
        \end{itemize}
      (Theorem D is a combination of Proposition \ref{msk_rank2}, and Corollary \ref{concorMtimesSk}.)
\end{manualtheorem}
Moreover, combining Theorem C and Theorem D, we obtain the following diffeomorphism classification result.
\begin{manualtheorem}{E}
 Let $M$ be a smooth projective plane-like manifold of dimension $8.$ 
  \begin{itemize}
      \item[(a)] Any closed, oriented, smooth manifold homeomorphic to $M \times \mathbb{S}^1$ is oriented diffeomorphic to either $M \times \mathbb{S}^1$ or $(M \times \mathbb{S}^1) \# \Sigma_{\mu},$ where $\Sigma_{\mu}$ is the exotic $9$-sphere corresponding to the $\text{Coker}(J_9)$ element $\mu.$
      \item[(b)] $bP_{18}\cap I_h(M\times\mathbb{S}^9)=\emptyset.$ 
  \end{itemize}
  (See Theorem \ref{classificationMsk} and Theorem \ref{bp18}.)
\end{manualtheorem}  
According to Hitchin and Lichnerowicz \cite{Hit74, AL63}, if a closed, spin, Riemannian manifold \( M \) admits a metric with positive scalar curvature, then its  \(\alpha\)-invariant \( \alpha(M) = 0 \), where
\[
\alpha: \Omega_{\text{spin}}^* \to KO^*
\]
is the ring homomorphism that associates a spin bordism class with the \(KO\)-valued index of the Dirac operator of a representative spin manifold. In particular, since \( \alpha(\Sigma_{\mu}) \neq 0 \) \cite{DAEBFP}, and the \(\alpha\)-invariant is additive on connected sums, the two non-diffeomorphic manifolds \( \mathbb{H}P^2 \times S^1 \) and \( (\mathbb{H}P^2 \times S^1) \# \Sigma_{\mu} \), given by Theorem E for \( M = \mathbb{H}P^2 \), exhibit different scalar curvature properties: the former admits a metric of positive scalar curvature, while the latter does not.

Finally, we remark that all the calculations in Theorems A, B, and D remain unchanged if, instead of \( M \times S^k \), we consider the product of \( M \) with any homotopy \( k \)-sphere.


	\subsection{Notation}
	\begin{itemize}
		\item Let $O_n$ be the orthogonal group, $PL_n$ is the simplicial group of piece-wise linear homeomorphisms of $\mathbb{R}^n$ fixing origin, $Top_n\subset O_n$ is the group of self homeomorphisms of $\mathbb{R}^n$ preserving origin, and $G_n$ be the set of homotopy equivalences. Denote by $O=\varinjlim~O_n$, $PL= \varinjlim PL_n$, $Top=\varinjlim Top_n,$ and $G=\varinjlim~ G_n$ \cite{kuiperlashof66,lashofrothenberg65,milgram}.
		\item Let $G/O$ be the homotopy fiber of the canonical map $BO\to BG$ between the classifying spaces for stable vector bundles and stable spherical fibrations \cite[\S2 and \S3]{milgram}, and $G/PL$ be the homotopy fiber of the canonical map $BPL\to BG$ between the classifying spaces for $PL$ $\mathbb{\R}^n$-bundles and stable spherical fibrations \cite{rudyak15}. Again $Top/O$ is the homotopy fibre of the map $BO\to BTop$ between the classifying spaces for stable vector bundles and stable topological $\mathbb{R}^n$-bundles \cite[Theorem 10.1 Essay IV]{kirby}. Similarly, $PL/O$ is thew homotopy fibre of the canonical map $BO\to BPL.$
		\item  There are standard fibre sequences $\cdots\to\Omega G/Top\xrightarrow{w} Top/O\xrightarrow{\psi} G/O\xrightarrow{\phi} G/Top$ and $O\xrightarrow{\Omega J}G\xrightarrow{j}G/O\xrightarrow{i}BO\xrightarrow{J}BG$ \cite{gw} where $J:\pi_n(SO)\to\pi_n^s$ is the classical $J$-homomorphism.
		\item We write $X_{(p)}$ for localization of the space $X$ at prime $p.$ We use the notation $f_{(p)}$ when a map $f$ is localized at $p.$ 
        \item If $J:\pi_n(O)\to \pi_n^s$ is non-trivial, We denote the quotient $\pi_n^s/{\mathrm{Im}(J)}$ by $\mathrm{Coker}(J_n).$ 
	\end{itemize}
	
	\subsection{Acknowledgements} 
      	The author expresses sincere gratitude to his thesis supervisor, Ramesh Kasilingam, for his invaluable guidance and support throughout the development of this paper. Much of the work in this paper emerged from productive discussions with him.  The author is grateful for his extensive help and discussions during the first revision of the paper. Additionally, the author thanks Ruizhi Huang and Samik Basu for their assistance in understanding the attaching map in the rank $\geq2$ case. Finally, the author is grateful to the referee for valuable comments and suggestions, which helped improve the clarity and quality of the paper.
	
	\section{Preliminaries}
	In this section, we recall some basic definitions and notations, which will be useful for our purpose. 
 \begin{defn}\label{definition_CM}
     Let $M$ be a smooth manifold. Let $(N,f)$ be a pair consisting of smooth manifold $N,$ together with a homeomorphism $f: N\to M.$ We say two pairs $(N_1,f_1)$ and $(N_2,f_2)$ are \textit{concordant} if there exists a diffeomorphism $g: N_1\to N_2$ and a homeomorphism $F: N_1\times [0,1]\to M\times[0,1]$ such that $F\vert_{N_1\times 0}=f_1$ and $F\vert_{N_1\times 1}= f_2\circ g.$ The set of all such \textit{concordance classes} is denoted by $\mathcal{C}(M).$ The concordance class of $(N,f)$ is denoted by $[(N,f)].$

     The canonical homeomorphism $h_{\Sigma}: M\#\Sigma\to M$ corresponding to homotopy $n$-sphere $\Sigma$ induces the class $[(M\#\Sigma,h_{\Sigma})]$ in $\mathcal{C}(M).$ We often write $[(M\#\Sigma, h_{\Sigma})]$ as $[M\#\Sigma].$
 \end{defn}
	
\begin{thm}{\rm (Kirby and Siebenmann, \cite[Page 194]{kirby})}\label{kirbyresult}
Let $M$ be a closed, oriented, smooth manifold of dimension $n\geq 5.$ Then there is a bijection between $\mathcal{C}(M)$ and $[M,Top/O].$
Furthermore, the concordance class of the given smooth structure on $M$ corresponds to the homotopy class of constant map under this bijection.
\end{thm}	

According to the identification in Theorem \ref{kirbyresult}, the degree one map $f_M: M\to \mathbb{S}^n$ induces the following map: 
\begin{align*}
    (f_M)^*: \mathcal{C}(\mathbb{S}^n)&\to \mathcal{C}(M)\\
    [\Sigma]&\mapsto [M\#\Sigma]
\end{align*}
Since $\mathcal{C}(\mathbb{S}^n)\cong \Theta_n\cong \pi_n(Top/O)$ for $n\geq 5,$ the \textit{concordance inertia group} $I_c(M)$ can be identified with $\mathrm{ker}\left(\Theta_n\xrightarrow{(f_{M})^*}\mathcal{C}(M)\right).$

\begin{note}
    There is an action of the self-homeomorphism group of $M$ on $\mathcal{C}(M).$ The orbit space of this action gives the set of orientation-preserving diffeomorphism classes of smooth manifolds that are homeomorphic to the given manifold $M.$
\end{note}

We recall the sequence 
  \begin{equation}\label{note2.3}
			0\to[ M\land N,Y]\xrightarrow{p^*}[M\times N,Y]\xrightarrow{i^*}[M\vee N,Y]\to0,
		\end{equation}
        induced from the cofibre $$M\vee N\xhookrightarrow{i}M\times N\xrightarrow{p} M\times N/M\vee N\simeq M\land N,$$ which is a split exact sequence for any $H$-space $Y.$ Next, we describe the concordance inertia group of $M\times\mathbb{S}^k$, assuming the stable homotopy type of $M$ is known, derived from the split short exact sequence \eqref{note2.3}.

\begin{lemma}{\rm \cite[Lemma 2.4]{SBRKAS}}\label{concordanceinertia}
Let $M$ be an $n$-dimensional, closed, oriented, and smooth manifold. Suppose $M$ is stably homotopy equivalent to $X\vee Z$, where $X$ is the cofibre of a certain map $f: \mathbb{S}^{n-1} \to \bigvee\limits_{i=1}^w A_i$ with $w\geq 2$, and the dimensions of both  $\bigvee\limits_{i=1}^w A_i$ and $Z$ are less than $n$. Then, for any integer $k\geq 1$ such that $n+k\geq 5,$ the concordance inertia group $I_c(M\times \mathbb{S}^k)$ is isomorphic to $$\mathrm{Im}\left(\bigoplus\limits_{i=1}^w [\Sigma^{k+1} A_i,Top/O]\xrightarrow{\bigoplus\limits_{i=1}^w (\Sigma^{k+1} f_i)^*}\pi_{n+k}(Top/O)\right),$$ where $f_i = Pr_i \circ f,~ 1\leq i\leq w$ are the compositions of $f$ with the projection maps $Pr_i$ from $\bigvee\limits_{i=1}^w A_i $ onto its $i^{th}$ factor.      
\end{lemma}
        
Theorem \ref{kirbyresult}, combined with the split short exact sequence (\ref{note2.3}), leads to the following corollary.
   \begin{cor}\label{concorMtimesSk}
       For any closed, oriented smooth manifold $M,$
       $$\mathcal{C}(M\times\mathbb{S}^k)\cong [M,Top/O]\oplus\pi_k(Top/O)\oplus[\Sigma^k M,Top/O],$$ \text{provided}~ $\mathit{dim}(M)+k\geq 5.$
   \end{cor}
\section{Smooth Structures on the Product of a \texorpdfstring{$3$}{}-connected \texorpdfstring{$8$}{}-Manifold with \texorpdfstring{$k$}{}-Sphere up to Concordance}
Let $M$ be a closed $3$-connected, smooth, $8$-dimensional manifold of rank $m\geq 1.$ According to \cite{wall_highly_connected_manifold}, $M$ is homotopy equivalent to a CW complex $(\bigvee\limits_{i=1}^m \mathbb{S}^4)\cup_{g}\mathbb{D}^8,$ where $g:\mathbb{S}^7\to \bigvee\limits_{i=1}^m \mathbb{S}^4$ is the attaching map of the top cell $\mathbb{D}^8$. From \cite{HDSW03}, if $m\geq2,$ then for any $l\geq 1$, the map $\Sigma^l g$ can be expressed as
\begin{equation}\label{stable}
  \Sigma^l g= \sum_{i=1}^m a_i \nu_{4+l},   
\end{equation}
where each $a_i\in \mathbb{Z}/{12}$, and $\nu_{4+l}$ is the generator of $\pi_{7+l}(\mathbb{S}^{4+l}).$ Applying the methods in \cite[\S 3]{rhuang},
\eqref{stable} can be further refined as
 \begin{equation*}
     \Sigma^l g= \mathrm{ind}(M) \nu_{4+l},
 \end{equation*}
  where $\mathrm{ind}(M)$ is the index the subgroup $\mathrm{Im}(\chi)$ in $\mathrm{Im}(J)$ if the homomorphism $\chi$ is nontrivial and is zero otherwise.
  
 This leads to the following stable homotopy type for $M$ \cite[Theorem 1.1]{rhuang}:
 \begin{equation}\label{stable_homotopy_type}
     M \simeq (\bigvee\limits_{i=1}^{m-1}\mathbb{S}^4)\vee Cone (\mathrm{ind}(M)\nu_4), ~\textit{if}~ m\geq 2.
 \end{equation}
 
For $m=1,$ it follows from \cite[\S 6]{EK62} that for $l\geq1,$
\begin{equation*}
    \Sigma^l g =t \nu_{4+l},\quad \text{where}~ t\in\mathbb{Z}/{24}.
\end{equation*}
Thus, for the rank 1 case, the stable homotopy type of $M$ is given by
\begin{equation}\label{stable_homotopy_type_rank1}
    M\simeq Cone(t\nu_4).
\end{equation}

Using the stable decompositions of $M$ given in \eqref{stable_homotopy_type}, and \ref{stable_homotopy_type_rank1}, it follows form Lemma \ref{concordanceinertia} that the concordance inertia group of $M\times\mathbb{S}^k$ is  
\begin{equation}\label{inertia_description}
    I_c(M\times\mathbb{S}^k)= \mathrm{Im}\left(\Theta_{5+k}\xrightarrow{(s \nu_{5+k})^*}\Theta_{8+k}\right), 
\end{equation}
where $s\in\mathbb{Z}/{24}$ is $\textrm{ind}(M)$ from \eqref{stable_homotopy_type} or $t$ from \eqref{stable_homotopy_type_rank1}.

The following lemma describes the image of the map $ \nu_l: \mathbb{S}^{l+3}\to\mathbb{S}^l$ at $Top/O$ level, where $\nu_l =\nu$ for all $l\geq 5$ \cite{toda}. 

\begin{lemma}\label{imagenu} 
Let $\nu_{l}:\mathbb{S}^{l+3}\to \mathbb{S}^l$ be the generator of $\pi_{l+3}(\mathbb{S}^l)=\pi_3^s.$ Then the image of the induced map $$( \nu_{l})^*:\pi_{l}(Top/O)\to\pi_{l+3}(Top/O)$$ is given in the following table: 
		\begin{center}
			\begin{tabular}{|c|c|c|c|c|c|c|c|c|c|c|c|c|c|}
				\hline
				{$l=$} & 7 &8 &9 & 10 & 11 & 12 & 13 & 14 & 15 & 16 & 17 & 18 & 19\\
				\hline
				{Im$((\nu_l)^*)$} & 0 & 0 & 0 & {$\z/3$} & 0 & 0 & 0 & {$\z/2$} & 0 & 0 & 0 & {$\z/2$} & $\z/2$\\
				\hline
			\end{tabular}
		\end{center} 
\end{lemma}

\begin{proof}

As $\pi_{12}(Top/O)$ and $\pi_{16}(Top/O)_{(3)}$ are both zero, the result clearly follows for $l=9, 12, 13.$

For $l=7$ and $11,$ consider the following commutative diagram 
\begin{center}
    \begin{tikzcd}
        &{\pi_l(Top/O)}\arrow[r,"\psi_*"]\arrow[d,"\nu_l^*"']&{0}\arrow[d]\\
        0\arrow[r]&{\pi_{l+3}(Top/O)}\arrow[r,"\psi_*"']&{\pi_{l+3}(G/O)}
    \end{tikzcd}
\end{center}
where $\psi_*:\pi_{l+3}(Top/O)\to \pi_{l+3}(G/O)$ is injective because $\pi_{l+4}(G/Top)$ is zero for $l=7$ and $11.$ It follows from the above diagram that the image of $\nu_l^*: \pi_{l}(Top/O)\to \pi_{l+3}(Top/O)$ is zero for $l=7$ and $11.$

Since $\pi_l(Top/O)\cong\pi_l(PL/O)$ for $l\geq 5,$ we show that the map $\nu_8^*:\pi_8(PL/O)\to\pi_{11}(PL/O)$ is the zero map. We use the following commutative diagram
\begin{center}
    \begin{tikzcd}
        {0}\arrow[r]&{\pi_8(O)}\arrow[r]\arrow[d,"\nu_8^*"']&{\pi_8(PL)}\arrow[r]\arrow[d,"\nu_8^*"]&{\pi_8(PL/O)}\arrow[r]\arrow[d,"\nu_8^*"]&{0}\\
        {0}\arrow[r]& {\pi_{11}(O)}\arrow[r]&{\pi_{11}(PL)}\arrow[r]&{\pi_{11}(PL/O)}\arrow[r]&{0}
    \end{tikzcd}
\end{center}
where the rows are induced from the fibration $PL/O\to BO\to BPL,$ and each row breaks into a short exact sequence \cite{hirsch}. Now it follows from the above commutative diagram that if the map $\nu_8^*: \pi_8(PL)\to\pi_{11}(PL)$ is zero, then $\nu_8^*: \pi_8(PL/O)\to \pi_{11}(PL/O)$ is also a zero map.

 The fibre sequence $G/PL\to BPL\to BG$ induces the following commutative diagram
\begin{center}
     \begin{tikzcd}
         &{\pi_8(PL)}\arrow[r,"\cong"]\arrow[d,"\nu_8^*"']&{\pi_8^s}\arrow[d,"0"]\\
         {\pi_{12}(G/PL)}\arrow[r]&{\pi_{11}(PL)}\arrow[r]&{\pi_{11}^s}\arrow[r]&{0}
     \end{tikzcd}
 \end{center} where  $\pi_n(PL)\cong \pi_n(G),$ if $n\equiv 0 \pmod{4}$ \cite[Lemma 4.2]{ramesh2}, and $\nu_8^*: \pi_8^s\to\pi_{11}^s$ is a zero map from \cite{toda}. Since $\pi_{11}(PL)\cong\z\oplus\z/8$ \cite{gw2} and the kernel of $\pi_{11}(PL)\to\pi_{11}^s$ is free abelian, $\nu_8^*: \pi_8(PL)\to\pi_{11}(PL)$ is a zero map.

Since $\pi_{13}(Top/O)\cong\z/3,$ we work locally at prime $3.$ As $\pi_{t}(Top/O)\cong \pi_{t}(G/O)\cong \pi_l^s$ for $t=10,13$, and the image of $(\nu_{10})_{(3)}=\alpha_1^*:\pi_{10}^s\to\pi_{13}^s$ is $\mathbb Z/3$ \cite{toda}, the result follows for $k=10.$

 We note that the exotic sphere in $\pi_{14}(Top/O)$ corresponds to the generator $\kappa$ or $\kappa+ \sigma^2$ in $\pi_{14}^s$ \cite[Page 45]{kubo}. Since $\nu\circ\sigma=0$ \cite[Theorem 14.1]{toda}, and $\nu\kappa \in \mathrm{Coker}(J_{17})$ \cite[Table A3.3]{ravenel}, the image of $\nu_{14}^*: \pi_{14}(Top/O)\to \pi_{17}(Top/O)$ is $\mathbb Z/2.$

We observe from the left part of the following commutative diagram: 
\begin{center}
    \begin{tikzcd}
        &{\pi_{15}(Top/O)}\arrow[r,"\psi_*"]\arrow[d,"\nu_{15}^*"']&{\pi_{15}(G/O)}\arrow[d,"\nu_{15}^*"]&{\pi_{15}^s}\arrow[l,"j_*"']\arrow[d,"\nu_{15}^*"]\\
       {0}\arrow[r]&{\pi_{18}(Top/O)}\arrow[r,"\psi_*"']&{\pi_{18}(G/O)}&{\pi_{18}^s}\arrow[l,"j_*"]
    \end{tikzcd}
\end{center} that if the map $\nu_{15}^*: \pi_{15}(G/O)\to\pi_{18}(G/O)$ is zero, then the image of $\nu_{15}^*:\pi_{15}(Top/O)\to\pi_{18}(Top/O)$ must also be zero. As $j_*(\eta\kappa)$ is non-trivial in $\pi_{15}(G/O)$ \cite[Table A3.3]{ravenel}, and $(\nu_{15})_{(2)}^*(\eta\kappa)=0$ \cite[Theorem 14.1]{toda}, the image of $\nu_{15}^*: \pi_{15}(G/O)\to\pi_{18}(G/O)$ is zero.

We note that $\pi_{16}(Top/O)\cong \mathrm{Coker}(J_{16})\cong \mathbb Z/2\{\eta^*\}$ \cite[Table A3.3]{ravenel} and $\langle \eta, \sigma, 2 \rangle =\eta^* ~\text{mod} ~\eta\rho$ \cite{toda, muk4}. Using the properties of the Toda brackets \cite[Page 33]{toda}, we have $\nu_{16}^*\eta^*=0.$ This implies that $\nu_{16}^*: \pi_{16}(Top/O)\to\pi_{19}(Top/O)$ is a zero map.

We have $\pi_{17}(Top/O)\cong bP_{18}\oplus \mathrm{Coker}(J_{17})\cong \pi_{18}(G/Top)\oplus\pi_{17}(G/O)\cong \bigoplus\limits_{i=1}^4 \mathbb Z/2.$ Since $\nu_{17}^*\circ\omega_*:\pi_{18}(G/Top)\to\pi_{20}(Top/O)$ is zero, and $\omega_*: \pi_{18}(G/Top)\to\pi_{17}(Top/O)$ is injective, the restriction of $\nu_{17}^*: \pi_{17}(Top/O)\to\pi_{20}(Top/O)$ to $\omega_*(\pi_{18}(G/Top))$ is zero. We have $\mathrm{Coker}(J_{17})\cong \mathbb Z/2\{\eta\eta^*\}\oplus \mathbb Z/2\{\nu\kappa\}\oplus \mathbb Z/2\{\Bar{\mu}\}$ where $\nu\kappa=\langle \eta\kappa,\eta,2\rangle$ \cite[Page 81]{muk4} and $\Bar{\mu}\in \langle \eta, 16\rho, 2\rangle$ \cite{muk66}. Now, the result follows from the properties of the Toda bracket and $\eta\circ\nu=0$ \cite{toda}.

Since $\pi_{18}(Top/O)\cong\pi_{18}(G/O)\cong\pi_{18}^s\cong \mathbb Z/2\{\eta\Bar{\mu}\}\oplus \mathbb Z/8\{\nu^*\}$ \cite{toda}, and $\nu\circ\nu^*= \sigma^3$, with $ j_*(\sigma^3)\neq 0$ in $\pi_{21}(G/O)$ \cite{ravenel}, along with the fact that $(\Bar{\mu}\circ \eta)\circ \nu =\Bar{\mu}\circ (\eta\circ\nu)=0$ \cite[Proposition 1.3]{GB67}, it follows that the image of $\nu_{18}^*: \pi_{18}(Top/O)\to\pi_{21}(Top/O)$ is $\mathbb Z/2.$

From \cite{gw2, GB70}, we obtain $\pi_{19}(Top/O)\cong bP_{20}\oplus \mathrm{Coker}(J_{19}).$ Now \cite[Theorem B]{JR72} indicates that the restriction of the Bredon pairing $\rho_{19,3}: \pi_{19}(Top/O)\times \pi_{3}^s\to\pi_{22}(Top/O)$ to $bP_{20}$ is trivial. 

By \cite[Corollary 2.2]{GB67}, there exists a map $\mathcal{F}: \mathrm{Coker}(J_{19})\times \pi_3^s\to\mathrm{Coker}(J_{22})$ such that the following diagram commutes:
\begin{center}
    \begin{tikzcd}
         {\pi_{19}(Top/O)\times\pi_3^s}\arrow[r,"p'\times id"]\arrow[d,"\rho_{19,3}"']& {\mathrm{Coker}(J_{19})\times \pi_3^s}\arrow[r]\arrow[d,"\mathcal{F}"]&{0}\\
         {\pi_{22}(Top/O)}\arrow[r,"p'"',"\cong"]&{\mathrm{Coker}(J_{22})}
    \end{tikzcd}
\end{center}
 and $\mathcal{F}(\Bar{\sigma},\nu)= \Bar{\sigma}\circ \nu,$ which is non-trivial by \cite{ravenel}. This implies that the image of the restriction of $\rho_{19,3}:\pi_{19}(Top/O)\times\pi_3^s\to\pi_{22}(Top/O)$ to $\mathrm{Coker}(J_{19})$ is $\mathbb Z/2.$

 Therefore, the image $\rho_{19,3}:\pi_{19}(Top/O)\times\pi_3^s\to\pi_{22}(Top/O)$ is $\mathbb Z/2$, and so the map $\nu_{19}^*: \pi_{19}(Top/O)\to\pi_{22}(Top/O)$ has image $\mathbb Z/2.$ 
\end{proof}

\begin{fact}\label{detection_at_3}
The non-triviality of the map $s(\nu_{5+k})_{(3)}$ is determined by the divisibility of the first Pontryagin class $p_1(M)$ by $3$ \cite[Remark 1.2]{TSO}. Thus, $3\mid p_1(M)$ implies that $s$ must be a multiple of $3.$ 
\end{fact}
Using \eqref{inertia_description}, Lemma \ref{imagenu} and the above fact, we determine $I_c(M\times\mathbb{S}^k)$ for certain values of $k$ in the following theorem.
\begin{thm}\label{ConcorMtimesSk}
Let $M$ be a closed, $3$-connected, smooth, $8$-manifold of rank $m\geq1.$ Then
\begin{itemize}
    \item[(i)] $I_c(M\times\s^k)=0$ for $1\leq k\leq 14,k \neq 5,9,13$ and $14.$ 
    \item[(ii)] For $k=5:$
    \begin{itemize}
        \item[(a)] $I_c(M\times\s^5)=0,$ if $3\mid p_1(M).$ 
        \item[(b)] $I_c(M\times\s^5)=\mathbb{Z}/3,$ if $3 \nmid p_1(M).$
    \end{itemize}
    \item[(iii)] For $k=9,13,14:$
    \begin{itemize}
      \item[(a)] If $m=1,$ then $I_c(M\times\mathbb{S}^k)=\mathbb{Z}/2.$
           \item[(b)] If $m\geq2,$ then
       \begin{enumerate}
           \item $I_c(M\times\mathbb{S}^k)=0$ for even $\textrm{ind}(M).$
           \item $I_c(M\times\mathbb{S}^k)=\mathbb{Z}/2$ for odd $\textrm{ind}(M)$.
       \end{enumerate}
    \end{itemize}
\end{itemize}
\end{thm}
\begin{proof}
We observe that if the rank of the manifold $M$ is $1,$ then the operation $Sq^4: H^4(M;\mathbb{Z}/2)\to H^8(M;\mathbb{Z}/2)$ is an isomorphism. Consequently, $(\Sigma^{k+1}g)_{(2)}$ is an odd multiple of $(\nu_{5+k})_{(2)}\in \mathbb{Z}/8.$ Combining this observation with Fact \ref{detection_at_3}, Lemma \ref{imagenu}, the results for $m=1$ follow from \eqref{inertia_description}.  

For $m\geq 2,$ the concordance inertia group of $M\times\mathbb{S}^k$ is determined from \eqref{inertia_description} using the Fact \ref{detection_at_3}, and Lemma \ref{imagenu}.
\end{proof}

From \cite{ravenel}, we observe that for $k=15,31,67,76,86,87$ and $96,$ there exists $x\in \mathrm{Coker}(J_{5+k})$ such that $(\nu_{5+k})_{(3)}\circ x =\alpha_1\circ x \neq 0 $ in $\mathrm{Coker}(J_{8+k}).$ From this, we deduce:
\begin{cor}
Let $M$ be a closed, $3$-connected, smooth, $8$-manifold with $3\nmid p_1(M).$ Then $\mathbb{Z}/3\subseteq I_c(M\times\mathbb{S}^k)$ for  $k=15,31,67,76,86,87$ and $96.$
\end{cor}
As $I_c(M\times\mathbb{S}^k)\subseteq I(M\times\mathbb{S}^k)\subseteq \Theta_{8+k},$ we have the following from Theorem \ref{ConcorMtimesSk}:
\begin{rmk}
    For any closed, $3$-connected, $8$-dimensional smooth manifold $M$ with $3\nmid p_1(M),$ the inertia group of $M\times\mathbb{S}^5$ is $\mathbb{Z}/3.$ In particular, $I(\mathbb{H}P^2\times\mathbb{S}^5)=\mathbb{Z}/3.$ 
\end{rmk}

Since $M\simeq \left(\bigvee\limits_{i=1}^m\mathbb{S}^4\right)\cup_{g}\mathbb{D}^8$, we have $[M, Top/O]= \pi_8(Top/O)$. Hence, by Corollary \ref{concorMtimesSk}, it suffices to compute  $[\Sigma^k M, Top/O]$ to determine $\mathcal{C}(M\times\s^k).$ Rather than finding $\mathcal{C}(M\times\mathbb{S}^k)$, we first focus on calculating $\mathcal{C}(\mathbb{H}P^2\times\mathbb{S}^k)$ and thereby computing $[\Sigma^k\mathbb{H}P^2, Top/O].$ To compute $[\Sigma^k\mathbb{H}P^2, Top/O],$ we repeatedly use the following long exact sequence \begin{equation}\label{longhP^2}
     \cdots\to\Theta_{5+k}\xrightarrow{\nu_{5+k}^*}\Theta_{8+k}\xrightarrow{(\Sigma^k f_{\h P^2})}[\Sigma^k \h P^2,Top/O]\xrightarrow{(\Sigma^k i)^*}\Theta_{4+k}\xrightarrow{\nu_{4+k}^*}\Theta_{7+k}\to\cdots
 \end{equation} induced from the cofibre sequence 
 \begin{equation}\label{cofihP^2}
     \cdots\to\s^{7+k}\xrightarrow{\nu_{4+k}}\s^{4+k}\xhookrightarrow{\Sigma^k i}\Sigma^k \h P^2\xrightarrow{\Sigma^k f_{\h P^2}}\s^{8+k}\xrightarrow{\nu_{5+k}}\s^{5+k}\to\cdots
 \end{equation}
 In the following proposition, we formulate $[\Sigma^k \h P^2,Top/O]$ for $1\leq k\leq 10.$
\begin{prop}\label{Concordancehp2sk}
\hspace{2em}
\begin{enumerate}
    \item[(i)] $[\Sigma^k\h P^2,Top/O]\cong\Theta_{8+k}$ for $k=1,2,8,10.$
    \item[(ii)] $[\Sigma^k\h P^2,Top/O]\cong\Theta_{4+k}$ for $k=4,5.$
    \item[(iii)] There exists a split short exact sequence $0\to\Theta_{11}\xrightarrow{(\Sigma^3 f_{\h P^2})^*}[\Sigma^3\h P^2,Top/O]\xrightarrow{(\Sigma^3 i)^*}\Theta_7\to0.$
    \item[(iv)] There exists a short exact sequence $0\to\mathbb{Z}/2\xrightarrow{(\Sigma^6 f_{\h P^2})^*}[\Sigma^6\h P^2,Top/O]\xrightarrow{(\Sigma^6 i)^*}\mathbb{Z}/2\to0,$ where $\mathbb{Z}/2=\Theta_{14}$ on the left side is the cokernel of $\nu_{11}^*$ and $\mathbb{Z}/2\subset \Theta_{10}$ on the right side is the kernel of $\nu_{10}^*.$ 
    \item[(v)] There exists a split short exact sequence $0\to\Theta_{15}\xrightarrow{(\Sigma^7 f_{\h P^2})^*}[\Sigma^7\h P^2,Top/O]\xrightarrow{(\Sigma^7 i)^*}\Theta_{11}\to0,$ where $\Theta_{11}$ is the kernel of $\nu_{11}^*.$
     \item[(vi)] There exists a split short exact sequence $0\to\mathbb{Z}/2\oplus\mathbb{Z}/2\oplus\mathbb{Z}/2\to[\Sigma^9\h P^2,Top/O]\to\mathbb{Z}/3\to0$ where $\z/2\oplus\z/2\oplus\z/2 \subset \Theta_{17}$ is the cokernel of $\nu_{14}^*$ and $\mathbb{Z}/3$ is the kernel of $\nu_{13}^*.$
\end{enumerate}
\end{prop}
\begin{proof}
The proofs of {(i)} and {(ii)} follow from the long exact sequence (\ref{longhP^2}) and Lemma \ref{imagenu}.

{(iii)} Since the images of $\nu_8^*:\Theta_8\to\Theta_{11}$ and $\nu_7^*:\Theta_7\to\Theta_{10}$ are both zero, we have the following short exact sequence from the long exact sequence (\ref{longhP^2}) 
\begin{equation}\label{eqn17}
 0\to\Theta_{11}\xrightarrow{(\Sigma^3 f_{\h P^2})^*}[\Sigma^3\h P^2,Top/O]\xrightarrow{(\Sigma^3 i)^*}\Theta_7\to0.   
\end{equation}
We now show that this short exact sequence splits. Consider the following ladder of exact sequences 
\begin{center}
    \begin{tikzcd}
             &&{0}\arrow[d]&{0}\arrow[d]\\
             &{0}\arrow[r]\arrow[d]&{[\Sigma^4\h P^2,Top/O]}\arrow[r,"\cong","(\Sigma^4 i)^*"']\arrow[d,"\psi_*"]&{\z/2}\arrow[r,"0"]\arrow[d,"\psi_*"]&{\z/{992}}\arrow[d]\\
             {\z/2\oplus\z/2}\arrow[r,"0"]\arrow[d]&{\z}\arrow[r,"(\Sigma^4 f_{\h P^2})^*"]\arrow[d,"\phi_*"']&{[\Sigma^4\h P^2,G/O]}\arrow[r,"(\Sigma^4 i)^*"]\arrow[d,"\phi_*"]&{\z\oplus\z/2}\arrow[r]\arrow[d,"\phi_*"]&{0}\arrow[d]\\
             {0}\arrow[r]\arrow[d]&{\z}\arrow[r,"(\Sigma^4 f_{\h P^2})^*"]\arrow[d,"\omega_*"']&{\z\oplus\z}\arrow[r,"(\Sigma^4 i)^*"]\arrow[d,"\omega_*"]&{\z}\arrow[r]\arrow[d,"\omega_*"]&{0}\arrow[d]\\
             {\z/2}\arrow[r,"0"]&{\z/{992}}\arrow[r,"(\Sigma^3 f_{\h P^2})^*"']\arrow[d]&{[\Sigma^3\h P^2,Top/O]}\arrow[r,"(\Sigma^3 i)^*"']\arrow[d]&{\z/{28}}\arrow[r,"0"]\arrow[d]&{\z/6}\\
             &{0}&{0}&{0}
    \end{tikzcd}
\end{center} where the columns are induced from the fibre sequence $\cdots\Omega(G/Top)\xrightarrow{\omega}Top/O\xrightarrow{\psi}G/O\xrightarrow{\phi}G/Top,$ the rows are induced from the cofibre sequence \eqref{cofihP^2}, and the second row from below splits as $Ext_{\z}^1(\z,\z)=0.$ From the top part of the diagram, we observe that $[\Sigma^4\h P^2, G/O]$ is either $\z\oplus\z$ or $\z\oplus\z\oplus\z/2.$ Since $\psi_*:[\Sigma^4\h P^2,Top/O]\to[\Sigma^4\h P^2,G/O]$ is injective, $[\Sigma^4\h P^2,G/O]\cong\z\oplus\z\oplus\z/2.$ As the image of $(\Sigma^4 f_{\h P^2})^*\circ \phi_*:\pi_{12}(G/O)\to[\Sigma^4 \h P^2, G/Top]$ is $992\z,$ and that of $\phi_*\circ (\Sigma^4 i)^*:[\Sigma^4 \h P^2, G/O]\to\pi_8(G/Top)$ is $28\z,$ it follows from the diagram above that the image of $w_*:[\Sigma^4\h P^2, G/Top]\to[\Sigma^3\h P^2, Top/O]$ is $\z/{992}\oplus\z/{28}.$ This implies that the short exact sequence \eqref{eqn17} splits.

{(iv)} The proof follows from the long exact sequence (\ref{longhP^2}) and Theorem \ref{ConcorMtimesSk}.

{(v)} From the long exact sequence (\ref{longhP^2}) and Theorem \ref{ConcorMtimesSk}, we obtain the following short exact sequence 
\begin{equation}\label{eqn18}
 0\to\Theta_{15}\xrightarrow{(\Sigma^7 f_{\h P^2})^*}[\Sigma^7\h P^2,Top/O]\xrightarrow{(\Sigma^7 i)^*}\Theta_{11}\to0  
\end{equation} 
This short exact sequence indicates that it is sufficient to work locally at prime $2.$ Since $\nu_{12}^*: \pi_{12}(G/O)\to\pi_{15}(G/O)$ is the zero map and $\pi_{2t+1}(G/Top)=0$ for all $t\geq 0$, the long exact sequence induced from the cofibre sequence \eqref{cofihP^2} along $G/Top$ and $G/O,$ respectively, yields the following short exact sequences: 
  \begin{equation}\label{sigmahp2gtop}
      0\to\z\to[\Sigma^8\h P^2,G/Top]\to\z\to0,
  \end{equation}
  and 
  \begin{equation}\label{sigmahp2go}
   0\to \mathbb{Z}\oplus\mathbb{Z}/2\to [\Sigma^8\mathbb{H}P^2, G/O]\to \mathbb{Z}\to 0.   
  \end{equation}
Clearly, both short exact sequences split.

Suppose \eqref{eqn18} does not split. Then this implies that $[\Sigma^7\mathbb{H}P^2, Top/O]_{(2)}$ is either $\mathbb{Z}/{2^6}\oplus\mathbb{Z}/{2^6}$ or $\mathbb{Z}/2\oplus\mathbb{Z}/{2^{11}}.$ Now consider the following commutative diagram where each row and column is an exact sequence induced from \eqref{cofihP^2} and fibre sequence $\cdots\Omega(G/Top)\xrightarrow{\omega}Top/O\xrightarrow{\psi}G/O\xrightarrow{\phi}G/Top$ respectively.
\begin{center}
    \begin{tikzcd}
             &{0}\arrow[d]&{0}\arrow[d]\\
             {0}\arrow[r]\arrow[d]&{\z/2}\arrow[r,"(\Sigma^8 f_{\h P^2})_{(2)}^*","\cong"']\arrow[d,"(\psi_*)_{(2)}"']&{[\Sigma^8\h P^2,Top/O]_{(2)}}\arrow[r,"(\Sigma^8 i)_{(2)}^*"]\arrow[d,"(\psi_*)_{(2)}"]&{0}\arrow[r]\arrow[d,"(\psi_*)_{(2)}"]&{\Theta_{15}}\arrow[d]\\
             {0}\arrow[r]\arrow[d]&{\z_{(2)}\oplus\z/2}\arrow[r,"(\Sigma^8 f_{\h P^2})_{(2)}^*"]\arrow[d,"(\phi_*)_{(2)}"']&{[\Sigma^8\h P^2,G/O]_{(2)}}\arrow[r,"(\Sigma^8 i)_{(2)}^*"]\arrow[d,"(\phi_*)_{(2)}"]&{\z_{(2)}}\arrow[r,"(\nu_{12})_{(2)}^*"]\arrow[d]&{\z/2}\arrow[d]\\
             {0}\arrow[r]\arrow[d]&{\z_{(2)}}\arrow[r,"(\Sigma^8 f_{\h P^2})_{(2)}^*"]\arrow[d,"(\omega_*)_{(2)}"']&{\z_{(2)}\oplus\z_{(2)}}\arrow[r,"(\Sigma^8 i)_{(2)}^*"]\arrow[d,"(\omega_*)_{(2)}"]&{\z_{(2)}}\arrow[r]\arrow[d,"(\omega_*)_{(2)}"]&{0}\arrow[d]\\
             {0}\arrow[r]\arrow[d]&{\z/2\oplus\z/{2^6}}\arrow[r,"(\Sigma^8 f_{\h P^2})_{(2)}^*"']\arrow[d,"(\psi_*)_{(2)}"']&{[\Sigma^7\h P^2,Top/O]_{(2)}}\arrow[r,"(\Sigma^7 i)_{(2)}^*"']\arrow[d,"(\psi_*)_{(2)}"]&{\z/{2^5}}\arrow[r,"0"]\arrow[d]&{\z/2}\\
             {\z_{(2)}}\arrow[r,"(\nu_{12})_{(2)}^*"']&{\z/2}\arrow[r,"(\Sigma^8 f_{\h P^2})_{(2)}^*"']\arrow[d]&{[\Sigma^7\h P^2,G/O]_{(2)}}\arrow[r,"(\Sigma^7 i)_{(2)}^*"']\arrow[d]&{0}\\
             &{0}&{0}
    \end{tikzcd}
\end{center}
Since \eqref{sigmahp2gtop} and \eqref{sigmahp2go} both split, we observe from the above diagram that the image of $(\phi_*)_{(2)}: [\Sigma^8\mathbb{H}P^2, G/O]\to[\Sigma^8\mathbb{H}P^2, G/Top]$ is $2^6\mathbb{Z}_{(2)}\oplus 2^5\mathbb{Z}_{(2)}.$ Thus, if $[\Sigma^7 \h P^2,Top/O]_{(2)} $ is either $\mathbb{Z}/{2}\oplus\mathbb{Z}/{2^{11}}$ or $\mathbb{Z}/{2^6}\oplus\mathbb{Z}/{2^6},$ then the image of $(\omega_*)_{(2)}: [\Sigma^8\mathbb{H}P^2, G/Top]_{(2)}\to[\Sigma^7\mathbb{H}P^2, Top/O]_{(2)}$ has order $2^6.$ However, this contradicts that $(\psi_*)_{(2)}:[\Sigma^7\mathbb{H}P^2, Top/O]_{(2)}\to [\Sigma^7\mathbb{H}P^2, G/O]_{(2)}\cong \mathbb{Z}/2$ is onto. Therefore, the short exact sequence \eqref{eqn18} splits. 

{(vi)} From Lemma \ref{imagenu}, we know that $\nu_{14}^*: \Theta_{14} \to \Theta_{17}$ has image $\mathbb{Z}/2$, while $\nu_{13}^* : \Theta_{13} \to \Theta_{16}$ is the zero map. Hence, from the long exact sequence \eqref{longhP^2}, we obtain the following short exact sequence: $$0\to\z/2\oplus\z/2\oplus\z/2\to[\Sigma^9 \h P^2,Top/O]\to\z/3\to0,$$ which clearly splits.
\end{proof}	

For any closed, $3$-connected, smooth $8$-dimensional manifold $M$, it is clear from the stable decompositions \eqref{stable_homotopy_type} and \eqref{stable_homotopy_type_rank1} that in order to compute $[\Sigma^k M, Top/O],$ it is sufficient to compute $[Cone(s\nu_{4+k}), Top/O],$ where $s\in \mathbb{Z}/{24}$ is given by $\textrm{ind}(M)$ from \eqref{stable_homotopy_type} or by $t$ from \eqref{stable_homotopy_type_rank1}. Note that $[Cone(s\nu_{4+k}), Top/O]$ fits into the following long exact sequence 
\begin{equation}\label{longM^8n}
    \cdots\to\Theta_{5+k}\xrightarrow{(s\nu_{5+k})^*}\Theta_{8+k}\xrightarrow{}[Cone(s\nu_{4+k}), Top/O]\xrightarrow{}\Theta_{4+k}\xrightarrow{(s\nu_{4+k})^*}\Theta_{7+k}\to\cdots
\end{equation} By applying the same method as in Proposition \ref{Concordancehp2sk}, we compute $[Cone(s\nu_{4+k}), Top/O]$ from the long exact sequence \eqref{longM^8n} using Theorem \ref{ConcorMtimesSk}. Combining the computation of $[Cone(s\nu_{4+k}), Top/O]$ for $1\leq k\leq 10$ with \eqref{stable_homotopy_type} for rank $m\geq2$ and \eqref{stable_homotopy_type_rank1} for rank $m=1$, we obtain the following:
\begin{prop}\label{msk_rank2}
    Let $M$ be a $3$-connected closed, smooth, $8$-manifold of rank $m\geq 1.$ Then
    \begin{itemize}
        \item[(a)] $[\Sigma^k M, Top/O]=\Theta_{8+k}$ for $k=1,2$ and $8.$
        \item[(b)] $[\Sigma^4 M, Top/O]= \bigoplus\limits_{i=1}^m\Theta_{8}.$ 
        \item[(c)] $[\Sigma^k M, Top/O]= \Theta_{8+k}\oplus \bigoplus\limits_{i=1}^m \Theta_{4+k}$   for $k=3$ and $7.$ 
       \item[(d)] For $k=5:$
       \begin{itemize}
           \item[(i)] If $3\mid p_1(M),$ then $[\Sigma^5 M, Top/O]= \Theta_{13}\oplus\bigoplus\limits_{i=1}^m \Theta_9.$
           \item[(ii)] If $3\nmid p_1(M),$ then $[\Sigma^5 M, Top/O]= \bigoplus\limits_{i=1}^m \Theta_9.$
       \end{itemize}
       \item[(e)] For $k=6:$
       \begin{itemize}
           \item[(i)] If $3\mid p_1(M),$ then $[\Sigma^6 M, Top/O]$ is either $\mathbb{Z}/2\oplus\mathbb{Z}/2\oplus\mathbb{Z}/3\oplus\bigoplus\limits_{i=1}^{m-1}\Theta_{10},$ or $\mathbb{Z}/4\oplus\mathbb{Z}/3\oplus\bigoplus\limits_{i=1}^{m-1}\Theta_{10}.$
           \item[(ii)] If $3\nmid p_1(M),$ then $[\Sigma^6 M, Top/O]$ is either $\mathbb{Z}/2\oplus\mathbb{Z}/2\oplus\bigoplus\limits_{i=1}^{m-1}\Theta_{10},$ or $\mathbb{Z}/4\oplus\bigoplus\limits_{i=1}^{m-1}\Theta_{10}.$
       \end{itemize}
        \item[(f)] For $k=9:$
        \begin{enumerate}
            \item If $m=1,$ then $[\Sigma^9 M, Top/O]= \bigoplus\limits_{i=1}^3\mathbb{Z}/2\oplus\Theta_{13}.$
            \item  If $m\geq 2:$
             \begin{itemize}
            \item[(i)] $[\Sigma^9 M, Top/O]= \Theta_{17}\oplus\bigoplus\limits_{i=1}^m\Theta_{13},$ for even $\textrm{ind}(M)$.
            \item[(ii)] $[\Sigma^9 M, Top/O]=\bigoplus\limits_{i=1}^3\mathbb{Z}/2\oplus\bigoplus\limits_{i=1}^m\Theta_{13},$ for odd $\textrm{ind}(M)$.
        \end{itemize}
        \end{enumerate}
        \item[(g)] For $k=10:$
        \begin{enumerate}
            \item When $m=1,~[\Sigma^{10}M, Top/O]=\Theta_{18}.$
            \item When $m\geq 2:$
               \begin{itemize}
        \item[(i)] If $\textrm{ind}(M)$ is even, then $[\Sigma^{10}M, Top/O]= H\oplus\bigoplus\limits_{i=1}^{m-1} \Theta_{14},$ where $H$ is an extension of the group $\Theta_{14}$ by the group $\Theta_{18}.$
        \item[(ii)] If $\textrm{ind}(M)$ is odd, then $[\Sigma^{10}M, Top/O]= \Theta_{18}\oplus\bigoplus\limits_{i=1}^{m-1} \Theta_{14}.$ 
        \end{itemize}
        \end{enumerate}
     
    \end{itemize}
\end{prop}

\section{Some applications}
Let $\Tilde{\pi}_{0}(Diff(M))$ denote the group of isotopy classes of orientation preserving self-diffeomorphism of $M.$ We know that there is a natural map $\gamma: \Theta_{n+1}\to \Tilde{\pi}_{0}(Diff(M))$ defined as follows: given $\Sigma\in \Theta_{n+1},$ it corresponds to a diffeomorphism $f: \mathbb{D}^n\to \mathbb{D}^n$ with $f\vert_{\mathbb{S}^{n-1}}=id,$ then $\gamma(\Sigma)\in Diff(M)$ with $\gamma(\Sigma)\vert_{M-\mathbb{D}^n}=id$ and $\gamma(\Sigma)\vert_{\mathbb{D}^n}=f.$ We now recall a lemma describing the inertia group of the product manifold $M\times\mathbb{S}^1.$

\begin{lemma}{\cite[Proposition 1]{JL70}}\label{inertia}
Let $M$ be a compact smooth manifold of dimension $n\geq 5.$ Then the inertia group of \( M \times \mathbb{S}^1 \) is given by  \[I(M\times\mathbb{S}^1)=\mathrm{ker}\left(\Theta_{n+1}\xrightarrow{\gamma} \Tilde{\pi}_{0}(Diff(M))\right).\]\\
Note that for any $\Sigma_1,\Sigma_2\in \Theta_{n+1},$ if $\gamma(\Sigma_1)$ is isotopic to $\gamma(\Sigma_2),$ then $(M\times\mathbb{S}^1)\#\Sigma_1$ is oriented diffeomorphic to $(M\times\mathbb{S}^1)\#\Sigma_2.$
\end{lemma}

Recall from \cite{DAEBFP} that an exotic $m$-sphere $\Sigma^m$ does not bound spin manifolds if and only if $m \equiv 1,2\ (\textrm{mod}\ 8), $ and $m>8,$ which is equivalent to the condition that $\alpha(\Sigma^m)=0,$ where $\alpha: \Theta_m\to KO^{-m}$ is the $\alpha$-invariant. In particular, $\Theta_9= bspin_{10}\oplus \mathbb{Z}/2\{\Sigma_{\mu}\},$ where $bspin_{10}$ denotes the group of homotopy $9$-spheres bounding spin manifolds and $\Sigma_{\mu}$ is the exotic $9$-sphere corresponding to the generator $\mu\in \mathrm{coker}(J_9).$ 

\begin{lemma}\label{Inertiagrouphp2s1}
Let $M$ be a smooth projective plane-like manifold of dimension $8.$ Then $I(M\times\mathbb{S}^1)=bspin_{10}=bP_{10}\oplus\mathbb{Z}/2\{\Sigma_{\nu^3}\},$ where $\Sigma_{\nu^3}$ is the exotic $9$-sphere corresponding to $\nu^3\in \mathrm{Coker}(J_9).$
\end{lemma}
\begin{proof}
  It is known that $\Theta_9\cong bP_{10}\oplus\mathbb{Z}/2\{\Sigma_{\mu}\}\oplus\mathbb{Z}/2\{\Sigma_{\nu^3}\},$ where $\Sigma_{\mu},\Sigma_{\nu^3}$ are exotic $9$-spheres associated to the generators $\mu, \nu^3 \in \mathrm{Coker}(J_9),$ respectively. According to \cite[Page 13]{SuWang2024}, the map $\gamma: \Theta_9\to \Tilde{\pi}_0(Diff(M))$ is surjective, and $\Tilde{\pi}_0(Diff(M))\cong \mathbb{Z}/2,$ generated by $\gamma(\Sigma_{\mu})$ \cite[Theorem 1.1]{SuWang2024}. Thus, we have $\mathrm{ker}\left(\Theta_9\xrightarrow{\gamma}\Tilde{\pi}_0(Diff(M))\right)\cong bP_{10}\oplus\mathbb{Z}/2\{\Sigma_{\nu^3}\},$ which, by Lemma \ref{inertia}, is the inertia group of $M\times\mathbb{S}^1$.
\end{proof}

\begin{thm}\label{classificationMsk}
    Let $M$ be a smooth projective plane-like $8$-dimensional manifold. Then any closed, oriented, smooth manifold $N$ homeomorphic to $M\times\mathbb{S}^1$ is oriented diffeomorphic to exactly one of the manifolds $M\times\mathbb{S}^1$ or $(M\times\mathbb{S}^1)\#\Sigma_{\mu}$, where $\Sigma_{\mu}$ is the exotic $9$-sphere associated with $\mu\in \mathrm{Coker}(J_9).$
\end{thm}
\begin{proof}
       Let $[(N,g)]$ be an element of $\mathcal{C}(M\times\mathbb{S}^1).$ If the normal invariant $\eta^{Diff}(g)$ is trivial, then, from the smooth surgery exact sequence, $N$ is (oriented) diffeomorphic to $(M\times\mathbb{S}^1)\#\Sigma$ for some $\Sigma\in bP_{10}.$ Since $bP_{10}\cap I(M\times\mathbb{S}^1)\neq \emptyset,$ the manifold $N$ is (oriented) diffeomorphic to $M\times\mathbb{S}^1.$ 

  Suppose $\eta^{Diff}(g)$ is non-trivial. From Corollary \ref{concorMtimesSk} and Proposition \ref{msk_rank2} {(a)}, we have $\mathcal{C}(M\times\mathbb{S}^1)\cong \Theta_9\oplus\Theta_8.$ If $[(N,g)]\in \Theta_9 \subset \mathcal{C}(M\times\mathbb{S}^1),$ then, by Lemma \ref{Inertiagrouphp2s1}, $N$ is diffeomorphic to either $M\times\mathbb{S}^1$ or $(M\times\mathbb{S}^1)\#\Sigma_{\mu}.$ On the other hand, if $[(N,g)]\in \Theta_8\subset \mathcal{C}(M\times\mathbb{S}^1),$ then $\eta^{Diff}(g)= [\bar{\nu}]\in\pi_8(G/O),$ where $\bar{\nu} $ is the generators of $\mathrm{Coker}(J_8)$ \cite[Theorem 1.1.14]{ravenel}. Since $[M, Top/O]=\Theta_8$ and $I(M)=\mathbb{Z}/2$ \cite{RK16, DCCN20}, there exists a self-homotopy equivalence $f:M\to M$ such that its normal invariant $\eta^{Diff}(f)=[\bar{\nu}]\in \pi_8(G/O).$ Since $M$ is simply connected, by \cite[Proposition 2.1]{PP99}, any self-homotopy equivalence of $M\times\mathbb{S}^1$ is diagonalizable. Hence, \cite[Proposition 2.3 and Theorem 2.5]{PP99} implies that $h=f\times id : M\times\mathbb{S}^1\to M\times\mathbb{S}^1$ is a self-homotopy equivalence of $M\times\mathbb{S}^1,$ and $\eta^{Diff}(h)=\eta^{Diff}(f).$ Now, the composition formula for normal invariant gives
     \begin{align*}
        \eta^{Diff}(h\circ g)= & \eta^{Diff}(h)+ (h^{-1})^* \eta^{Diff}(g)\\
         = & [\bar{\nu}]\pm [\bar{\nu}]=0.
    \end{align*}
    Therefore, by the previous case, $N$ is (oriented) diffeomorphic to $M\times\mathbb{S}^1.$ This completes the proof.  
\end{proof}

\begin{lemma}\label{Inertiagroupop2s1}
         Let $M$ be a smooth projective plane-like manifold of dimension $16.$ Then $I(M\times\mathbb{S}^1)=bP_{18}\oplus\mathbb{Z}/2\{\Sigma_{\eta\eta^*}\},$ where $\Sigma_{\eta\eta^*}$ is the exotic $17$-sphere associated with the generator $\eta\eta^*\in \mathrm{Coker}(J_{17}).$
\end{lemma}
\begin{proof}
We note from \cite[Page 13 and 14]{SuWang2024} that $\Theta_{17}/{\text{ker}(\gamma)}\cong \mathbb{Z}/2\{\Sigma_{\bar{\mu}}\}\oplus\mathbb{Z}/2\{\Sigma_{\nu\kappa}\},$ with $\Sigma_{\bar{\mu}}$ and $\Sigma_{\nu\kappa}$ being exotic $17$-spheres corresponding to the elements $\bar{\mu}, \nu\kappa \in \mathrm{Coker}(J_{17}).$ Since $\Theta_{17}$ fits into the split short exact sequence $$0\to bP_{18}\to \Theta_{17}\to \mathrm{Coker}(J_{17})\to 0,$$ it follows that $\text{ker}(\gamma)\cong bP_{18}\oplus \mathbb{Z}/2\{\Sigma_{\eta\eta^*}\},$ where $\Sigma_{\eta\eta^*}$ is the exotic $17$-sphere associated with $\eta\eta^*\in \mathrm{Coker}(J_{17}).$ By Lemma \ref{inertia}, this implies that $I(M\times\mathbb{S}^1)=bP_{18}\oplus \mathbb{Z}/2\{\eta\eta^*\}.$  
\end{proof}

\begin{thm}\label{bp18}
    $bP_{18}\cap I_h(M\times\mathbb{S}^9)= \emptyset$ for any smooth projective plane-like manifold $M$ of dimension $8.$  
\end{thm}
\begin{proof}
    From the smooth surgery exact sequence of $M\times\mathbb{S}^9,$ it is sufficient to show that the surgery obstruction map $\Theta^{Diff} : [\Sigma(M\times\mathbb{S}^9), G/O]\to L_{18}(e)$ is trivial. The surgery obstruction map  $\Theta^{Diff} : [\Sigma(M\times\mathbb{S}^9), G/O]\to L_{18}(e)$ fits into the following commutative diagram:
   
\begin{center}
    \begin{tikzcd}
        {[\Sigma(M\times\mathbb{S}^9), G/O]}\arrow[r,"\Theta^{Diff}"]\arrow[d,"p^*"']&{L_{18}(e)}\arrow[r,"\omega^{Diff}"]\arrow[d,"\cong"]&{\mathcal{S}^{Diff}(M\times\mathbb{S}^9)}\\
        {[M\times\mathbb{S}^9\times\mathbb{S}^1, G/O]}\arrow[r,"\bar{\Theta}^{Diff}"']&{L_{18}(\mathbb{Z})}
    \end{tikzcd}
\end{center}
 where $p^*:[\Sigma (M\times\mathbb{S}^9), G/O]\to [M\times\mathbb{S}^9\times\mathbb{S}^1, G/O]$ is injective by \eqref{note2.3}, and $L_{18}(e)\to L_{18}(\mathbb{Z})$ is an isomorphism from \cite{JS69}. Let $f: \Sigma(M\times\mathbb{S}^9)\to G/O,$ then from the above commutative diagram, the surgery obstruction of $f$ is given by \cite{gw} 
\begin{equation}
       \Theta^{Diff}(f)=\bar{\Theta}^{Diff}(p^*(f))=\langle v_{4}^2(M\times\mathbb{S}^9\times\mathbb{S}^1) (p^*(f))^*(\mathcal{K}_2), [\mathbb{H}P^2\times\mathbb{S}^9\times\mathbb{S}^1]\rangle=0,\nonumber\\
\end{equation}
where $\mathcal{K}_2\in H^2(F/O)=\mathbb{Z}/2$ is the generator, the $4$-th Wu class $v_4(M\times\mathbb{S}^9\times\mathbb{S}^1)$ of $M\times\mathbb{S}^9\times\mathbb{S}^1$ is zero from \cite[Lemma 3.5]{LK2003}and \cite{MS}. This completes the proof.
\end{proof}

Combining \cite[Lemma 9.1]{Kawakubo69}, Theorem \ref{ConcorMtimesSk} {(iii)}, and Theorem \ref{bp18}, we obtain the following result.
\begin{prop}
 For any $8$-dimensional smooth projective plane-like manifold $M$, the homotopy inertia group $I_h(M\times\mathbb{S}^9)$ is either $\mathbb{Z}/2\{\Sigma_{\nu\kappa}\}$ or $\mathbb{Z}/2\{\Sigma_{\nu\kappa}\}\oplus\mathbb{Z}/2\{\Sigma_{\eta\eta^*}\},$ where $\Sigma_{\nu\kappa}, \Sigma_{\eta\eta^*}$ are the exotic $17$-spheres associated with the respective generators $\nu\kappa$ and $\eta\eta^*$ of $\mathrm{Coker}(J_{17})$.
\end{prop}

\end{document}